\documentclass[reqno,11pt]{amsart}

\usepackage{graphicx}
\usepackage{url}
\usepackage[cmtip,arrow]{xy}  
\usepackage{pb-diagram,pb-xy}  
\usepackage{graphicx}    
\usepackage{amsfonts}
\usepackage{amsmath, amssymb}
\usepackage[latin1]{inputenc}
\usepackage{color, enumerate}

\usepackage{pgfplots}
\usepackage{subfigure}
\usepackage{caption}
\usepackage{float}

\usepackage{amsmath}
\usepackage{amsthm}
\usepackage{amsfonts}
\usepackage{array}
\newcolumntype{C}[1]{>{\centering\arraybackslash}p{#1}}

\newcommand{\SF}{\varepsilon}

\renewcommand{\L}{\mathcal L}
\newcommand{\Y}{\mathcal Y}
\newcommand{\U}{\mathcal{U}}

\newcommand{\T}{\mathbb{T}}
\newcommand{\F}{\mathcal{F}}

\renewcommand{\SF}{{\mathcal{F}}}

\newcommand{\SL}{{\mathcal{L}}}

\newcommand{\SU}{{\mathcal{U}}}

\newcommand{\Z}{\mathbb{Z}}

\newcommand{\R}{\mathbb{R}}

\renewcommand{\S}{\mathbb{S}}

\newcommand{\GL}{\operatorname{GL}}
\newcommand{\Id}{\operatorname{Id}}

\newcommand{\const}{\textup{const}}

\newcommand{\SSO}{\textup{SO}}

\newcommand{\we}{\wedge}
\newcommand{\pp}[2]{\frac{\partial#1}{\partial#2}}

\newcommand{\bd}{\partial}

\newcommand{\frbd}[1]{\frac{\bd}{\bd #1}}

\newcommand{\so}{{\mathfrak s\mathfrak o}}

\newtheorem{proposition}{Proposition}
\newtheorem{theorem}[proposition]{Theorem}
\newtheorem{definition}[proposition]{Definition}
\newtheorem{lemma}[proposition]{Lemma}

\newtheorem{example}[proposition]{Example}

\theoremstyle{remark}
\newtheorem{remark}[proposition]{Remark}

\begin{document}

\bibliographystyle{alpha}

\title[Non-commutative $b$-integrable systems]{Non-commutative integrable systems on $b$-symplectic manifolds}

\author{Anna Kiesenhofer}
\address{Department of Mathematics, Universitat Polit\`ecnica de Catalunya, EPSEB, Avinguda del Doctor Mara\~{n}\'{o}n 44--50, Barcelona, Spain}
\email{anna.kiesenhofer@upc.edu}
\author{Eva Miranda}
\address{Deparment of Mathematics,
 Universitat Polit\`{e}cnica de Catalunya and BGSMath, EPSEB, Avinguda del Doctor Mara\~{n}\'{o}n 44--50, Barcelona, Spain}
\email{eva.miranda@upc.edu}
\thanks{Anna Kiesenhofer is supported by AGAUR FI doctoral grant. Anna Kiesenhofer and Eva Miranda are  supported by the  Ministerio de Econom\'{\i}a y Competitividad project with reference MTM2015-69135-P (MINECO-FEDER) and by the Generalitat de Catalunya project with reference 2014SGR634 (AGAUR) }

\date{\today}
\begin{abstract}
In this paper we study non-commutative integrable systems on  $b$-Poisson manifolds.  One important source of examples (and motivation) of such systems comes from considering non-commutative systems on manifolds with boundary having the right asymptotics on the boundary.  In this paper we describe this and other examples and we prove an action-angle theorem for non-commutative integrable systems on a  $b$-symplectic manifold in a neighbourhood of a Liouville torus inside the critical set of the Poisson structure associated to the $b$-symplectic structure.
\end{abstract}
\maketitle


\section{Introduction}
A non-commutative integrable system on a symplectic manifold with boundary yields a non-commutative system on a class of Poisson manifolds called $b$-Poisson manifolds, as long as the asymptotics of the system satisfy certain conditions near the boundary.  $b$-Poisson manifolds constitute a class of Poisson manifolds which recently has been   studied extensively(see for instance \cite{Guillemin2011}, \cite{Guillemin2012}, \cite{Guillemin2013} and \cite{gualtierietal}) and integrable systems on such manifolds have been the object of study in \cite{KMS}, \cite{km} and \cite{DKM}.

 In \cite{Laurent-Gengoux2010}  an action-angle coordinate for Poisson manifolds is proved on a neighbourhood of a regular Liouville torus. This theorem cannot be applied to a neighborhood of a Liouville torus contained inside the critical set of the Poisson structure where the rank of the bivector field is no longer maximal. In this paper we extend the techniques in \cite{Laurent-Gengoux2010} to consider  a neighbhourhood of a Liouville torus inside the critical set of a $b$-Poisson manifolds thus proving an action-angle theorem for non-commutative systems on $b$-Poisson manifolds.

The action-angle theorem for non-commutative integrable systems for symplectic manifolds was proved by Nehoroshev in \cite{nekhoroshev}. Our proof follows a combination of techniques from   \cite{Laurent-Gengoux2010} with techniques native to $b$-symplectic geometry. As in \cite{Laurent-Gengoux2010} the key point of the proof is to find a torus action attached to a non-commutative integrable system and extend the Darboux-Carath\'{e}odory coordinates in a neighbourhood of the invariant subset. The upshot is the use of $b$-symplectic techniques and toric actions on these manifolds \cite{Guillemin2013}, \cite{gmps2} as we did in \cite{KMS} and \cite{km} for commutative systems on $b$-manifolds. The proof is a combination of the theory of torus actions with a refinement of the commutative proof by considering Cas-basic forms and working with them as a subcomplex of the $b$-De Rham complex. The action-angle theorem for commutative integrable systems on $b$-symplectic manifolds  yields semilocal models as twisted cotangent lifts (see \cite{km}). It is also possible to visualize the action-angle theorem for non-commutative systems using twisted cotangent lifts.

The organization of this paper is as follows: In Section 2 we introduce the basic tools that will be needed in this paper. In Section 3 we provide a list of examples which includes non-commutative systems on symplectic manifolds with boundary and examples obtained from group actions including twisted b-cotangent lifts. We end this section exploring the Galilean group as a source of non-commutative examples in $b$-symplectic manifolds. In Section 4 we state and prove  the action-angle coordinate theorem for $b$-symplectic manifolds.
\section{Preliminaries}\label{sec:pre}

\subsection{Integrable systems and action-angle coordinates on Poisson manifolds}

A Poisson manifold is a pair $(M, \Pi)$ where $\Pi$ is a bivector field such that the associated bracket on functions
$$\{f,g\}:=\Pi(df,dg), \quad f,g:M\to \R$$
satisfies the Jacobi identity.
The Hamiltonian vector field of a function $f$ is defined as
$X_f := \Pi(df,\cdot)$. This allows us to formulate equations of motion just as in the symplectic setting, i.e. given a Hamiltonian function $H$ we consider the flow of the vector field $X_H$. The concept of integrable systems is well understood in the symplectic context. A similar definition is possible in the Poisson setting and the famous Arnold-Liouville-Mineur theorem on the semilocal structure of integrable systems has its analogue in the Poisson context. Both commutative and non-commutative integrable systems on Poisson manifolds were studied in \cite{Laurent-Gengoux2010}.

\begin{definition}[Non-commutative integrable system on a Poisson manifold]\label{def:noncomintsyspoisson}
 Let $(M,\Pi)$ be a Poisson manifold of (maximal) rank $2r$. An $s$-tuple of functions
  $F=(f_1,\dots,f_s)$ on $M$ is a {\bf non-commutative (Liouville) integrable system} of rank $r$ on $(M,\Pi)$ if
\begin{enumerate}
  \item[(1)] $f_1,\dots,f_s$ are independent (i.e. their differentials are independent on a dense open subset of $M$);
  \item[(2)] The functions $f_1,\dots,f_r$ are in involution with the functions $f_1,\dots,f_s$;
  \item[(3)] $r+s =\dim M$;
  \item[(4)] The Hamiltonian vector fields of the functions $f_1, \dots,f_r$ are linearly independent at some point of $M$.
\end{enumerate}
  Viewed as a map, $F:M\to\R^s$ is called the {\bf momentum map} of $(M,\Pi,F)$.
\end{definition}

 When all the integrals commute, i.e. $r=s$, then we are dealing with the conventional case of a commutative integrable system.

\begin{example}[A generic example]
 Consider  the manifold $\T^r \times \R^{s}$  with coordinates
 $$(\theta_1,\dots,\theta_r,p_1,\dots,p_{r},z_1 ,\dots, z_{s-r} )$$
 equipped with the Poisson structure
 $$ \Pi=\sum_{i=1}^r \pp{}{\theta_i}\we\pp{}{p_i} + \pi'$$
 where $\pi'$ is any Poisson structure on  $\R^{s-r}$. Then the functions
 $$(p_1,\ldots,p_r,z_1,\ldots,z_s)$$
 define a non-commutative integrable system of rank $r$.
\end{example}

 As we will see in Theorem \ref{thm:action-angle_intro} below, any non-commutative integrable system semilocally takes this form, more precisely in the neighborhood of a  regular compact connected level set of its integrals $(f_1,\ldots,f_s)$.

\subsubsection{Standard Liouville tori}\label{par:tori}

Let $(M,\Pi,F)$ be a non-commutative integrable system of rank $r$.
We denote the non-empty subset of $M$ where the differentials $d f_1,\dots,d f_s$ (resp.\ the Hamiltonian vector
fields $X_{f_1}, \dots,X_{f_r}$) are independent by~$\SU_F$ (resp.\ $M_{F,r}$).

On the non-empty open subset $M_{F,r}\cap\U_F $ of~$M$, the Hamiltonian vector fields
$X_{f_1},\dots, X_{f_r}$ define an integrable distribution of rank $r$ and hence a foliation ${\mathcal F} $ with $r$-dimensional leaves, see \cite{Laurent-Gengoux2010}.

We will only deal with the case where $\SF_m$ is compact.  Under
this assumption, $\SF_m$ is a compact $r$-dimensional manifold, equipped with $r$ independent commuting vector
fields, hence it is diffeomorphic to an $r$-dimensional torus $\T^r$. The set $\SF_m$ is called a \emph{standard
  Liouville torus} of $F $.

The action-angle coordinate theorem  proved in  \cite{Laurent-Gengoux2010} (Theorem 1.1) gives a semilocal description of the Poisson structure around a standard Liouville torus of a non-commutative integrable system:

\begin{theorem}[{\bf Action-angle coordinate theorem for non-commutative integrable systems on Poisson manifolds}]\label{thm:action-angle_intro}
  Let $(M,\Pi,F)$ be a non-commutative integrable system of rank
  $r$, where $F=(f_1,\dots,f_s)$ and suppose that
  $\SF_m$ is a standard Liouville torus, where $m\in M_{F,r}\cap\U_F$.  Then there exist ${\R} $-valued smooth
  functions $(p_1,\dots, p_r,z_1, \dots, z_{s-r} )$ and $ {\R}/{\Z}$-valued smooth functions
  $({\theta_1},\dots,{\theta_r})$, defined in a neighborhood $U$ of $\SF_m$, and functions
  $\phi_{kl}=-\phi_{lk}$, which are independent of $\theta_1,\dots,\theta_r,p_1,\dots,p_r$, such that
  \begin{enumerate}
    \item The functions $(\theta_1,\dots,\theta_r,p_1,\dots,p_{r},z_1 ,\dots, z_{s-r} )$ define a diffeomorphism
            $U\simeq\T^r\times B^{s}$;
    \item The Poisson structure can be written in terms of these coordinates as,
    \begin{equation*}
       \Pi=\sum_{i=1}^r \pp{}{\theta_i}\we\pp{}{p_i} + \sum_{k,l=1}^{s-r} \phi_{kl}(z) \pp{}{z_k}\we\pp{}{z_l} ;
    \end{equation*}
    \item The leaves of the surjective submersion $F=(f_1,\dots,f_{s})$ are given by the projection onto the
      second component $\T^r \times B^{s}$, in particular, the functions $f_1,\dots,f_s$ depend on
      $p_1,\dots,p_r,z_1, \dots, z_{s-r}$ only.
  \end{enumerate}
  The functions $\theta_1,\dots,\theta_{r}$ are called \emph{angle coordinates}, the functions $p_1,\dots,p_r$
  are called \emph{action coordinates} and the remaining coordinates $z_1,\dots,z_{s-r}$ are called
  \emph{transverse coordinates}.
\end{theorem}



\subsection{$b$-Poisson and $b$-symplectic manifolds}

A symplectic form $\omega$ induces a Poisson structure $\Pi$ defined via
$$\Pi(df,dg)= \omega(X_f,X_g)$$
where $X_f, X_g$ are the Hamiltonian vector fields defined with respect to $\omega$. On the other hand, a Poisson structure which does not have full rank everywhere, i.e. the set of Hamiltonian vector fields spans the tangent space at every point, does not induce a symplectic structure.  However, if the Poisson structure drops rank in a controlled way as defined below, it is possible to associate a so-called $b$-symplectic structure.

\begin{definition}[$b$-Poisson structure]\label{definition:firstb}
Let $(M^{2n},\Pi)$ be an oriented Poisson manifold. If the map
$$p\in M\mapsto(\Pi(p))^n\in\bigwedge^{2n}(TM)$$
is transverse to the zero section, then $\Pi$ is called a \textbf{$b$-Poisson structure} on $M$. The hypersurface $Z=\{p\in M|(\Pi(p))^n=0\}$ is the \textbf{critical hypersurface} of $\Pi$. The pair $(M,\Pi)$ is called a \textbf{$b$-Poisson manifold}.
\end{definition}

It is possible and convenient to work in the ``dual" language of forms instead of bivector fields. The object equivalent to a $b$-Poisson structure will be a $b$-symplectic structure. To define $b$-symplectic structures and, in general, $b$-forms we introduce the concept of $b$-manifolds and the $b$-tangent bundle associated to the critical set $Z$:

\begin{definition}
A \textbf{$b$-manifold} is a pair $(M,Z)$ of an oriented manifold $M$ and an oriented hypersurface $Z\subset M$.
A \textbf{$b$-vector field} on a $b$-manifold $(M,Z)$ is a vector field which is tangent to $Z$ at every point $p\in Z$.
\end{definition}

The set of $b$-vector fields is a Lie subalgebra of the algebra of all vector fields on $M$. Moreover, if $x$ is a local defining function for $Z$ on some open set $U\subset M$ and $(x,y_1,\ldots,y_{N-1})$ is a chart on $U$, then the set of $b$-vector fields on $U$ is a free $C^\infty(M)$-module with basis $(x \frbd{x}, \frbd{y_1},\ldots, \frbd{y_N}). $
A locally  $C^\infty(M)$-module  has a vector bundle associated to it. We call the vector bundle associated to the sheaf of $b$-vector fields the \textbf{$b$-tangent bundle} denoted $^b TM$. The \textbf{$b$-cotangent bundle}  $^b T^*M$ is, by definition, the vector bundle dual to $^b TM$.

Given a defining function $f$ for $Z$, let $\mu\in\Omega^1(M\setminus Z)$ be the one-form $\frac{df}{f}$. If $v$ is a $b$-vector field then the pairing $\mu(v)\in C^\infty(M\setminus Z)$ extends smoothly over $Z$ and hence $\mu$ itself extends smoothly over $Z$ as a section of $^b T^*M$.
We will write $\mu=\frac{df}{f}$, keeping in mind that on $Z$ the expression only makes sense when evaluated on $b$-tangent vectors.

\begin{definition}[$b$-de Rham-$k$-forms]
The sections of the vector bundle $\Lambda^k(^b T^*M)$ are called   \textbf{$b$-$k$-forms} ($b$-de Rham-$k$-forms) and the sheaf of these forms is denoted $^b\Omega^k(M)$.
\end{definition}

For $f$ a defining function of $Z$ every $b$-$k$-form can be written as
\begin{equation}\label{eq:bDeRham}
\omega=\alpha\wedge\frac{df}{f}+\beta, \text{ with } \alpha\in\Omega^{k-1}(M) \text{ and } \beta\in\Omega^k(M).
\end{equation}

The decomposition \eqref{eq:bDeRham} enables us to extend the exterior $d$ operator to $^b\Omega^k(M)$ by setting
$$d\omega=d\alpha\wedge\frac{df}{f}+d\beta.$$
The right hand side is well defined and agrees with the usual exterior $d$ operator on $M\setminus Z$ and also extends smoothly over $M$ as a section of $\Lambda^{k+1}(^b T^*M)$. Since we have $d^2=0$, we can define the differential complex of $b$-forms, the $b$-de Rham complex.

\begin{definition}
Let $(M^{2n},Z)$ be a $b$-manifold and $\omega\in\,^b\Omega^2(M)$ a closed $b$-form. We say that $\omega$ is \textbf{$b$-symplectic} if $\omega_p$ is of maximal rank as an element of $\Lambda^2(\,^b T_p^* M)$ for all $p\in M$.
\end{definition}

It was shown in \cite{Guillemin2012} that $b$-symplectic and $b$-Poisson manifolds are in one-to-one correspondence. \\

The classical Darboux theorem for symplectic manifolds has its analogue in the $b$-symplectic case:

\begin{theorem}[{\bf $b$-Darboux theorem} \cite{Guillemin2012}]\label{thm:bdarboux}
Let $(M,Z, \omega)$ be a $b$-symplectic manifold. Let $p \in Z$ be a point and $z$ a local defining function for $Z$. Then, on a neighborhood of $p$ there exist coordinates $(x_1,y_1,\dots ,x_{n-1},y_{n-1}, z, t)$ such that
$$\omega=\sum_{i=1}^{n-1} dx_i\wedge dy_i+\frac{1}{z}\,dz\wedge dt.$$
\end{theorem}

The cohomology of the $b$-de Rham complex, whose groups are denoted by  $^b H^*(M)$, can be understood from the classic de Rham cohomologies of $M$ and $Z$ via the Mazzeo-Melrose theorem:

\begin{theorem}[{\bf{Mazzeo-Melrose}}]\label{thm:mazzeomelrose}
The $b$-cohomology groups of $M^{2n}$ satisfy
$$^b H^*(M)\cong H^*(M)\oplus H^{*-1}(Z).$$
\end{theorem}

Under the Mazzeo-Melrose isomorphism, a $b$-form of degree $p$ has two parts: its first summand, the \emph{smooth} part, is determined (by Poincar\'{e} duality) by integrating the form along any $p$-dimensional cycle transverse to $Z$ (such an integral is improper due to the singularity along $Z$, but the principal value of this integral is well-defined). The second summand, the \emph{singular} part, is the residue of the form along $Z$.

\subsection{$b$-functions}

 It is convenient to enlarge the set of smooth functions to the set of \textbf{$b$-functions} $^b C^\infty(M)$, so that the $b$-form $\frac{df}{f}$ is exact, where $f$ is a defining function for $Z$. We define a $b$-function to be a function on $M$ with values in $\R \cup \{\infty\}$ of the form
$$c\,\textrm{log}|f| + g,$$
 where $c \in \mathbb{R}$ and $g$ is a smooth function. For ease of notation, from now on we identify $\R$ with the completion $\R \cup \{\infty\}$.

 We define the differential operator $d$ on this space in the obvious way:
$$d(c\,\textrm{log}|f| + g):= \frac{c \, df}{f} + d g \in\, ^b\Omega^1(M),$$
where $d g$ is the standard de Rham derivative.

As in the smooth case, we define the ($b$-)Hamiltonian vector field of a $b$-function $f\in ^b C^\infty(M)$ as the (smooth) vector field $X_f$ satisfying
$$\iota_{X_f} \omega = -df.$$

Obviously, the flow of a $b$-Hamiltonian vector field preserves the $b$-symplectic form and hence the Poisson structure, so $b$-Hamiltonian vector fields are in particular Poisson vector fields.

\subsection{Twisted $b$-cotangent lift}\label{sec:twisted}

Given a Lie group action on a smooth manifold $M$,
$$\rho:G\times M \to M:(g,m)\mapsto \rho_g(m)$$
 we define the cotangent lift of the action to $T^\ast M$ via the pullback:
$$\hat{\rho}:G\times ^b \! T^\ast M \to ^b \! T^\ast M : (g,p)\mapsto \rho^\ast_{g^{-1}}(p).$$
It is well-known that the lifted action $\hat \rho$ is Hamiltonian with respect to the canonical symplectic structure on $T^\ast M$ (see \cite{guilleminandsternberg}).

We want to view the lifted action as a $b$-Hamiltonian action by means of a construction first described in \cite{km}.

Consider  $T^\ast S^1$ with standard coordinates $(\theta, a)$. We endow it with the following one-form defined for $a\neq 0$, which we call the logarithmic Liouville one-form in analogy to the construction in the symplectic case:
$\lambda_{tw,c} = \log |a| d\theta$ for $a\neq 0$.

Now for any $(n-1)$-dimensional manifold $N$, let $\lambda_N$ be the classical Liouville one-form on $T^\ast N$. We endow the product $T^\ast (S^1 \times N) \cong T^\ast S^1 \times T^\ast N$ with the product structure $\lambda:= (\lambda_{tw,c}, \lambda_N)$ (defined for $a\neq 0$). Its negative differential $\omega = -d\lambda$ extends to a $b$-symplectic structure on the whole manifold and the critical hypersurface is given by $a=0$.

Let $K$ be a Lie group acting on $N$ and consider the component-wise action of $G:=S^1\times K$ on $M:=S^1 \times N$ where $S^1$ acts on itself by rotations. We lift this action to $T^\ast M$ as described above. This construction, where $T^\ast M$ is endowed with the $b$-symplectic form $\omega$, is called the \textbf{twisted $b$-contangent lift}.

If $(x_1,\ldots,x_{n-1})$ is a chart on $N$ and $(x_1,\ldots,x_{n-1},y_1,\ldots,y_{n-1})$ the corresponding chart on $T^\ast N$ we have the following local expression for $\lambda$
\begin{equation*}\label{logliouvilleform}
\lambda = \log |a| d\theta + \sum_{i=1}^{n-1} y_i dx_i .
\end{equation*}

Just as in the symplectic case, this action is Hamiltonian with moment map given by contracting the fundamental vector fields with the Liouville one-form $\lambda$.

\section{Non-commutative $b$-integrable systems}\label{sec:noncommutativeb}

In  \cite{KMS} we introduced a definition of integrable systems for $b$-symplectic manifolds, where we allow the integrals to be $b$-functions. Such a ``$b$-integrable system'' on a $2n$-dimensional manifold consists of $n$ integrals, just as in the symplectic case. Here we introduce the definition for the more general non-commutative case:

\begin{definition}[Non-commutative $b$-integrable system]\label{def:bintsys}
A  non-commutative $b$-integrable system of rank $r$ on a $2n$-dimensional $b$-symplectic manifold $(M^{2n},\omega)$ is an $s$-tuple of functions  $F=(f_1,\ldots,f_r,f_{r+1},\ldots, f_s)$ where $f_1,\ldots,f_r$ are $b$-functions and $f_{r+1},\ldots, f_s$ are smooth such that the following conditions are satisfied:

\begin{enumerate}
  \item[(1)] The differentials $df_1,\ldots,df_s$ are linearly independent as $b$-cotangent vectors on a dense open subset of $M$ and on a dense open subset of $Z$;
  \item[(2)] The functions $f_1,\dots,f_r$ are in involution with the functions $f_1,\dots,f_s$;
  \item[(3)] $r+s =2n$;
    \item[(4)] The Hamiltonian vector fields of the functions $f_1, \dots,f_r$ are linearly independent as smooth vector fields at some point of $Z$.
\end{enumerate}
\end{definition}

We call the first $r$ functions $(f_1,\ldots,f_r)$ the commuting part of the system and the last $s-r$ functions the non-commuting part.

The case $r=s=n$ where we are dealing with a commutative system was studied in \cite{KMS}.

We denote the non-empty subsets of $M$ where condition (1) resp. (4) are satisfied by~$\SU_F$ resp.\ $M_{F,r}$. The points of the intersection $M_{F,r}\cap\U_F $ are called {\it regular}. As in the general Poisson case, the Hamiltonian vector $X_{f_1},\ldots,X_{f_r}$ fields define an integrable distribution of rank $r$ on this set and we denote the corresponding foliation by $\F$. If the leaf through a point $m\in M$ is compact, then it is an $r$-torus (``{\bf Liouville torus}''), denoted $\F_m$.

\begin{remark}In the symplectic case, if the differentials $df_i (i=1,\ldots,r)$ are linearly independent at a point $p$, then also the corresponding Hamiltonian vector fields $X_{f_i}$ are independent at $p$. However, the situation is more delicate in the $b$-symplectic case. The differentials $df_i$ are $b$-one-forms. At a point $p$ where the $df_i$ are independent as $b$-cotangent vectors, the corresponding Hamiltonian vector fields $X_{f_i}$ are independent at $p$ as $b$-tangent vectors. However, for $p\in Z$ the natural map $^b TM|_p \to TZ|_p$ is not injective and therefore we cannot guarantee independence of the $X_{f_i}$ as smooth vector fields. This is why the condition (4) is needed. As an example, consider $\R^2$ with standard coordinates $(t,z)$ and $b$-symplectic structure
$$\frac{1}{t} dt \wedge dz.$$
Then the function $z$ has a differential $dz$ which is non-zero at all points of $\R^2$, but the Hamiltonian vector field of $z$ is $t \frac{\partial}{\partial t}$ and vanishes along $Z=\{ t = 0\}$. We do not allow this kind of systems in our definition, since we are interested precisely in the dynamics on $Z$ and the existence of $r$-dimensional Liouville tori there. We remark that the definition has already been given in an analogous way for general Poisson manifolds in \cite{Laurent-Gengoux2010}.
\end{remark}

\section{Examples of (non-commutative) $b$-integrable systems}\label{subsec:examples}
%

\subsection{Non-commutative integrable systems on manifolds with boundary}\label{manifoldswithboundary}
In \cite{KMS} we introduced new examples of integrable systems using existing examples on manifolds with boundary. We can reproduce a similar scheme in the non-commutative case.  As a concrete example, let the manifold with boundary be $M = N  \times H_+$, where $(N, \omega_N)$ is any symplectic manifold and $H_+$ is the upper hemisphere including the equator. We endow the interior of $H_+$ with the symplectic form $\frac{1}{h} dh \wedge d\theta$, where $(h,\theta)$ are the standard height and angle coordinates and the interior of $M$ with the corresponding product structure. Now let  $(f_1, \dots, f_s)$ be a non-commutative integrable system of rank $r$ on $N$. Then on the interior of $M$ we can, for instance, define the following (smooth) non-commutative integrable system:
$$(\log|h|, f_1, \dots, f_s )$$
Taking the double of $M$ we obtain a non-commutative $b$-integrable system on $N \times S^2$.

\subsection{Examples coming from $b$-Hamiltonian $\T^r$-actions}

In \cite{bolsinov} it is shown how to construct integrable systems from the Hamiltonian action of a Lie group $G$ on a {\it symplectic} manifold $M$: Let $\mu: M \to \mathfrak{g}^\ast$ be the moment map of the action and consider the algebra of functions on $M$ generated by $\mu$-basic functions and $G$-invariant functions. Then under certain assumptions, this algebra is {\it complete} in the sense of \cite{bolsinov}, Definition 1.1 therein. This result is the content of Theorem 2.1 in \cite{bolsinov}. In our terminology, this means that the algebra of functions admits a basis of functions $f_1,\ldots,f_s$ which form a non-commutative integrable system on $M$. The assumptions needed for this to hold are satisfied in particular when the action is proper, which is the case for any compact Lie group $G$.

This result can be used in the $b$-symplectic case to semilocally construct a non-commutative $b$-integrable system on a $b$-symplectic manifolds $M^{2n}$ with an effective Hamiltonian $\T^r$-action as follows: Let us denote the critical hypersurface of $M$ by $Z$ and  assume $Z$ is connected. Let $t$ be a defining function for $Z$. A Hamiltonian  $\T^r$-action on a $b$-symplectic manifold, by definition, satisfies that the $b$-one-form  $\iota_{X^\#}\omega$ is exact for all $X \in\mathfrak{t}$. We consider an action with the property that, moreover, for some $X \in\mathfrak{t}$ the $b$-one-form  $\iota_{X^\#}\omega$ is a genuine $b$-one-form, i.e. not smooth. Then the following  proposition proved in \cite{Guillemin2013} about the ``splitting" of the action holds: The critical hypersurface $Z$ is a product $\SL \times \S^1$, where $\SL$ is a symplectic leaf inside $Z$ and in a neighborhood of $Z$ there is a splitting of the Lie algebra $\mathfrak{t}\simeq\mathfrak{t}_{Z}\times\langle X \rangle$, which induces a splitting $\mathbb{T}^r\simeq\mathbb{T}^{r-1}_{Z}\times\mathbb{S}^1$ such that the $\mathbb{T}^{r-1}_{Z}$-action on $Z$ induces a Hamiltonian $\mathbb{T}^{r-1}_{Z}$-action on $\mathcal{L}$.  Let $\mu_{\mathcal{L}}:\mathcal{L}\to\mathfrak{t}^*_{Z}$ be the moment map of the latter. Then on a neighborhood $\mathcal{L}\times\mathbb{S}^1\times(-\varepsilon,\varepsilon)\simeq\mathcal{U}\subset M$ of $Z$ the $\T^r$-action has moment map
\begin{equation*}\begin{aligned}
\mu_{\mathcal{U}\setminus Z}:\mathcal{L}\times\mathbb{S}^1\times((-\varepsilon,\varepsilon)\setminus\{0\})&\to\mathfrak{t}^*\simeq\mathfrak{t}^*_{Z}\times\mathbb{R}\\
(\ell,\rho,t)&\mapsto(\mu_{\mathcal{L}}(\ell),c\log|t|).
\end{aligned}\end{equation*}
Let $(f_1,\ldots,f_s)$ be the non-commutative integrable system induced on $\SL$ by applying the theorem in \cite{bolsinov} to the  $\T^{r-1}$-action on $\SL$. This system has rank $r-1$. On a neighborhood $\SL \times \{-\delta < \theta < \delta\} \times  \{-\epsilon < t < \epsilon\}$ it extends to a non-commutative $b$-integrable system $(\log|t|,f_1,\ldots,f_s)$ of rank $r$. The Liouville tori of the system are the orbits of the action.

\subsection{The geodesic flow}

A special case of a $\T^r$-action is obtained in the case of a Riemannian manifold $M$ which is assumed to have the property that all its geodesics are closed. These manifolds are called P-manifolds. In this case the geodesics admit a common period (see e.g. \cite{besse}, Lemma 7.11); hence their flow induces an $S^1$-action on $M$. In the same way the standard cotangent lift induces a system on $T^*M$  we can use the twisted $b$-cotangent lift (see subsection \ref{sec:twisted}) to obtain a $b$-Hamiltonian $S^1$-action on $T^\ast M$ and hence a non-commutative $b$-integrable system on $T^\ast M$.
In dimension two, examples of P-manifolds are Zoll and Tannery surfaces (see Chapter 4 in \cite{besse}).

\subsection{The Galilean group}

The Galilean group has its physical origin in the (non-relativistic) transformations between two reference frames which differ by relative motion at a constant velocity $b$. Together with spatial rotations and translations in time and space, this is the so-called {\it (inhomogeneous) Galilean group} $G$. We now present in detail this example as a non-commutative integrable system, see also \cite{davidandeva}.

We consider the evolution space
$$V=\R \times \R^3 \times \R^3 \ni (t,x,y),$$
where $t\in \R $ is time and $x,y \in \R^3$ are the position and velocity respectively.

The Galilean group can be viewed as a Lie subgroup of $\GL(\R, 5)$
consisting of matrices of the form
\begin{equation}\label{gal}\left( \begin{array}{ccc}
A & b & c \\
0 & 1 & e \\
0 & 0 & 1 \end{array} \right),\quad  A\in \SSO(3), b \in \R^3, c \in \R^3, e \in \R.
\end{equation}
If we denote the matrix above by $a$ then the action $a_V$ of the Galilean group on $V$ is defined as follows:
$$a_V(t,x,v)=(t^\ast, x^\ast,y^\ast)$$
where $t^\ast=t+e$, $x^\ast=Ax+bt+c$, $y^\ast=A y +b$.

%
The Lie algebra $\mathfrak{g}$ of $G$ is given by the set of matrices \cite{souriau}:
$$\left( \begin{array}{ccc}
j(\omega) & \beta & \gamma \\
0 & 0 & \epsilon \\
0 & 0 &  0\end{array} \right), \qquad \epsilon \in \R, \omega \in \R^3, \beta \in \R^3, \gamma \in \R^3.$$

Here, $j$ is the map that identifies $\R^3$ with $\so(3)$. Now instead of letting $G$ act on the evolution space $\R^7$, we  consider the action on the ``space of motions" $\R^3 \times \R^3$, which is obtained by fixing time, $t=t_0$. This space is symplectic with the canonical symplectic form and the action of $G$ on it is Hamiltonian.

In the literature the following integrals of the action are considered \cite{souriau}: Consider the basis of $\mathfrak{g}$ given by the union of the standard basis on each of its components $\so(3)$, $\R^3$ (corresponding to spatial translation $\gamma$), $\R$ (corresponding to time translation $\epsilon$) and the Galilei boost Lie algebra $\R^3$ (corresponding to the shift in velocity $\beta$). The corresponding integrals are, respectively, the components of the angular momentum $J = x \times y$, velocity vector $y$ and position vector $x$ and the energy $E$. This system is non-commutative.



We want to investigate the action of certain subgroups of $G$ and construct $b$-versions of the integrable systems. We will consider the space of motions $\R^6$ with coordinates $(x,y)$ as described above and time $t=0$.

\paragraph{\bf Subgroup given by $A=\Id$}
First, consider the subgroup of matrices of the form \eqref{gal} where $A$ is the identity matrix $\Id\in \SSO(3)$. Then we have an action of $\R^6$ on itself; in coordinates $(x,y)$ as above the action consists of shifts in the $x$ and $y$ directions. This action is Hamiltonian with moment map  and given by the full set of coordinates $(x_1,x_2,x_3,y_1,y_2,y_3)$. Clearly, this defines a non-commutative integrable system (of rank zero).

\paragraph{\bf Subgroup $\SSO(3)\times \R^3$}
Now let $c,e$ be constant; for the sake of simplicity we assume they are equal to zero. Consider  the subgroup of $G$ where only $A \in \SSO(3)$ and $b\in \R^3$ vary. Then the action on $\R^6$ is given by
\begin{equation}\label{so3r3}
A\cdot(x,y)=(Ax,Ay+b).
\end{equation}

First we want to see that the $\SSO(3)$-action is Hamiltonian. Consider the standard basis of the Lie algebra $\mathfrak{s}\mathfrak{o}(3)$ corresponding under $j$ to the unit vectors in $\R^3$:
$$e_1 =\left( \begin{array}{ccc}
0 & 0 & 0 \\
0 & 0 & -1 \\
0 & 1 & 0 \end{array} \right)  ,
e_2 = \left( \begin{array}{ccc}
0 & 0 & 1 \\
0 & 0 & 0 \\
-1 & 0 & 0 \end{array} \right) ,
e_3 =
\left( \begin{array}{ccc}
0 & -1 & 0 \\
1 & 0 & 0 \\
0 & 0 & 0 \end{array} \right) . $$
On $\R^3$ they describe rotations around the $x_1$, $x_2$- and $x_3$-axis respectively. The corresponding fundamental vector fields on $\R^6$ are
\begin{align*}
e_1^\# &=   x_3 \pp{}{x_2}-y_2 \pp{}{y_3}-x_2 \pp{}{x_3}+y_3 \pp{}{y_2},\\
e_2^\# &= x_1 \pp{}{x_3}-y_3 \pp{}{y_1}- x_3 \pp{}{x_1}+y_1 \pp{}{y_3} ,\\
e_3^\# &= x_2 \pp{}{x_1}-y_1 \pp{}{y_2}-x_1 \pp{}{x_2}+ y_2 \pp{}{y_1}.
\end{align*}
One checks that these vector fields are Hamiltonian with respect to the following functions:
$$f_1 = x_2y_3 - x_3 y_2,\quad f_2 = x_3 y_1  - x_1 y_3, \quad f_3 = x_1 y_2 - x_2 y_1. $$
Note that the $f_i$ are the components of angular momentum $J=x\times y$. Hence we have seen that the $\SSO(3)$-action is Hamiltonian.
The commutators are:
$$\{f_1, f_2\}=\omega(X_{f_1},X_{f_2}) = x_1 y_2 - x_2 y_1 = f_3,$$
and similarly $\{f_2, f_3\}=f_1$ and $\{f_3, f_1\}=f_2$.

Since the $f_i$ do not commute we need additional functions to define an integrable system on $\R^6$. This is where the $\R^3$ action, given by the parameter $b$ in Equation \eqref{so3r3} comes into play. It has fundamental vector fields $\frac{\partial}{\partial y_i}$ and the corresponding Hamiltonian functions are the coordinates $x_i$. Together with the integrals $f_i$ they form a non-commutative integrable system $(f_1,f_2,f_3,x_1,x_2,x_3)$ of rank zero.

\paragraph{\bf Subgroup $\S^1 \times \R^3 \times \R^3$}

Above we have studied the $\SSO(3)$ action on $\R^6$. Now we restrict to the $\S^1$-subgroup of $\SSO(3)$ given by rotations around the $x_1$- and $y_1$-axis. The associated integral is $f_1 = x_2 y_3 - x_3 y_2$.  To obtain a non-commutative integrable system of non-zero rank, we can e.g. add the functions $x_2, x_3, y_2$, which do not commute with $f_1$, and the function $y_1$, which commutes with all the other functions. Hence we have obtained a  non-commutative integrable system $(y_1,f_1,x_2,x_3,y_2)$ of rank one.

\paragraph{\bf Some $b$-versions of these constructions}

We view $\R^6$ as a $b$-symplectic manifold with critical hypersurface given by $Z=\{y_1=0\}$ and canonical $b$-symplectic structure
$$\frac{dy_1}{y_1}\wedge dx_1 + \sum_{i=2}^r dy_i\wedge dx_i.$$ We want to see if the actions  of the subgroups  above can be seen as Hamiltonian actions on the $b$-symplectic manifold $\R^6$ (i.e. their fundamental vector fields are Hamiltonian with respect to the $b$-symplectic structure). We treat the above cases one by one:
\begin{itemize}
\item The system $(x_1,x_2,x_3,y_1,y_2,y_3)$ translates into the non-commutative $b$-integrable system $(x_1,x_2,x_3,\log|y_1|,y_2,y_3)$, i.e. the Hamiltonian vector fields with respect to the $b$-symplectic structure are the same and the system fulfils the required independence and commutativity properties.
\item The $\SSO(3)\times \R^3$ action with moment map $(f_1,f_2,f_3,x_1,x_2,x_3)$ is {\it not} Hamiltonian with respect  to the $b$-symplectic structure. Indeed, away from $Z$, the fundamental vector field of the $\SSO(3)$-action above associated to the Lie algebra element $e_2$ has Hamiltonian function
$$x_3 \log |y_1| - x_1 y_3,$$
but this does not extend to a $b$-function on $\R^6$.
\item The system $(y_1,f_1,x_2,x_3,y_2)$ translates into the non-commutative $b$-integrable system $(\log|y_1|,f_1,x_2,x_3,y_2)$; the induced action is the same as in the smooth case. On the other hand, the smooth system where we replace $y_1$ by $x_1$, i.e. $(x_1,f_1,x_2,x_3,y_2)$, does not have such an analogue in the $b$-setting. Indeed, with respect to the $b$-symplectic form, the Hamiltonian vector field of the first function $x_1$ is $y_1$ and vanishes on $Z$, so the Hamiltonian vector fields of these functions are nowhere independent on $Z$.
\end{itemize}

%
%
%
%

\section{Action-angle coordinates for non-commutative $b$-integrable systems}\label{ch:bint}

In Theorem 8 we recalled the action-angle coordinate theorem for non-commutative integrable systems on Poisson manifolds, which was proved in \cite{Laurent-Gengoux2010}. For $b$-symplectic manifolds and the commutative $b$-integrable systems defined there, we have proved an action-angle coordinate theorem \cite{KMS}, which is similar to the symplectic case in the sense that even on the hypersurface $Z$ where the Poisson structure drops rank there is a foliation by Liouville tori (with dimension equal to the rank of the system) and a semi-local neighborhood with ``action-angle coordinates'' around them. The main goal of this paper is to establish a similar result in the non-commutative case, proving the existence of $r$-dimensional invariant tori on $Z$ and action-angle coordinates around them.

\subsection{Cas-basic functions}\label{sec:cas}

Consider a non-commutative $b$-integrable system $F$ on any Poisson manifold $(M,\Pi)$, where we denote the Poisson bracket by $\{\cdot,\cdot\}$. Let $V:=F(M)\cap \R^s$ be the ``finite" target space of the integrals $F$. If we want to emphasize the functions $F$ we are referring to, we will also write $V_F$.
The space  $V$ inherits a Poisson structure $\{\cdot,\cdot\}_{V}$ satisfying the following property:
$$\{g ,h \}_{V}\circ F =\{g \circ F,h \circ F \}, $$
where $g,h$ are functions on $V$.
Note that the values of the brackets $\{f_i,f_j\}$ on $M$ uniquely define the Poisson bracket $\{\cdot,\cdot\}_{V}$.

An $F$-basic function on $M$ is a function of the form $g\circ F$. The Poisson structure $\{\cdot,\cdot\}_{V}$ allows us to define the following important class of functions:

\begin{definition}[Cas-basic function]
 An $F$-basic function $g\circ F$ is called {\bf Cas-basic} if $g$ is a Casimir function with respect to $\{\cdot,\cdot\}_{V}$, i.e. the Hamiltonian vector field of $g$ on $V$ is zero.
\end{definition}

We recall the following characterisation of Cas-basic functions proved in \cite{Laurent-Gengoux2010} in the setting of integrable systems on Poisson manifolds. The proof in the $b$-case is the same.

\begin{proposition}
A function is Cas-basic if and only if it commutes with all $F$-basic functions.
\end{proposition}

\subsection{Normal forms for non-commutative $b$-integrable systems}\label{sec:normalforms}

\begin{definition}[Equivalence of non-commutative $b$-integrable systems]
Two  non-commutative $b$-integrable systems $F$ and $F'$ are equivalent if there exists a {\it Poisson} map
$$\mu:V_F\to V_{F'}$$
   taking one to the other:  $F' = \mu \circ F$. Here, $\mu$ is a Poisson map with respect to the Poisson structures induced on $V_F$ and $V_{F'}$ as defined in the previous section.
\end{definition}

We will not distinguish between equivalent systems: if the action-angle coordinate theorem that we will prove holds for one system then it holds for all equivalent systems too.

We prove a first ``normal form" result for non-commutative $b$-integrable systems:
\begin{proposition}\label{prop:normalform}
 Let $(M,\omega)$ be a $b$-symplectic manifold of dimension $2n$ with critical hypersurface $Z$. Given a non-commutative $b$-integrable system $F=(f_1,\ldots,f_s)$ of rank $r$ there exists an equivalent non-commutative $b$-integrable system of the form $(\log|t|,f_2,\ldots, f_s)$ where $t$ is a defining function of $Z$ and the functions $f_2,\ldots,f_s$ are smooth.
\end{proposition}
\begin{proof}
 First, assume that one of the functions $f_1,\ldots,f_r$ is a genuine $b$-function, without loss of generality $f_1 = g+c \log|t'|$ where $c\neq 0$ and $t'$ a defining function of $Z$. Dividing $f_1$ by the constant $c$ and replacing the defining function $t'$ by $t:=e^g t'$, we can restrict to the case $f_1=\log|t|$. We subtract an appropriate multiple of $f_1$ from the other functions $f_2,\ldots,f_r$ so that they become smooth. Note that this does not affect their independence nor the commutativity condition for $f_1,\ldots,f_{r}$, since $f_1$ commutes with all the integrals. Also, since these operations do not affect the non-commutative part of the system, the induced Poisson bracket on the target space (cf. Section \ref{sec:cas}) remains unchanged. Hence  we have obtained an equivalent $b$-integrable system of the desired form.

If all the functions $f_1,\ldots,f_s$ are smooth then from the independence of  $df_i$ ($i=1,\dots,s$) as $b$-one-forms on the set of regular points $\U_F \cap M_{F,r}$ it follows that
\begin{equation}\label{subm}
 df_1 \wedge \ldots \wedge df_s \wedge dt \neq 0 \in \Omega_p^s \quad \text{for } p \in \U_F \cap M_{F,r},
\end{equation}
where $t$ is a defining function of $Z$. Therefore the functions $f_1,\ldots,f_s, t$ define a submersion on $\U_F \cap M_{F,r}$ whose level sets are $(r-1)$-dimensional. On the other hand, the Hamiltonian vector fields $X_{f_1},\ldots,X_{f_r}$ are linearly independent (on $\U_F \cap M_{F,r}$) and tangent to the leaves of this submersion, because $f_1,\ldots,f_r$ commute with all $f_j, j=1,\ldots, s$ and also with $t$, since any Hamiltonian vector field is tangent to $Z$.
Contradiction.
\end{proof}

\begin{remark}
 Recall that the Liouville tori of a non-commutative $b$-integrable system $F$ are, by definition, the leaves of the foliation induced by $X_{f_i},i=1,\ldots,r$ on $\U_F \cap M_{F,r}$. A Liouville torus that intersects $Z$ lies inside $Z$, since the Hamiltonian vector fields are Poisson vector fields and therefore tangent to $Z$. Moreover, since at least one of the first $r$ integrals $f_1,\ldots,f_r$ has non-vanishing ``$\log$" part, the Liouville tori inside $Z$ are {\it transverse} to the symplectic leaves.
\end{remark}

We now prove a normal form result which holds semilocally around a Liouville torus. It describes the topology of the system: we will see that semilocally the foliation of Liouville tori is a product $\T^r \times B^s$, but the result does not yet give information about the Poisson structure.

\begin{proposition}\label{prop:standarddonut}
Let $m\in Z$ be a regular point of a non-commutative $b$-integrable system $(M,\omega, F)$. Assume that the integral manifold $\SF_m$ through $m$ is compact (i.e. a torus $\T^r$). Then there exist a neighborhood $U\subset \U_F \cap M_{F,r}$ of $\SF_m$ and a diffeomorphism
$$\phi:U\simeq \T^r \times B^{s},$$
which takes the foliation $\SF$ induced by the system to the trivial foliation $\{\T^n\times \{b \}\}_{b\in B^n}$.
\end{proposition}

\begin{proof}
As described in the previous proposition, we can assume that our system has the form $(\log|t|,f_2,\ldots,f_s)$ where $f_2,\ldots,f_s$ are smooth. Consider the submersion
$$\tilde F:= (t, f_2,\ldots,f_s): \SU_F \to \R^s$$
which has $r$-dimensional level sets. The Hamiltonian vector fields $X_{f_1},\ldots,X_{f_r}$ are tangent to the level sets. By comparing dimensions we see that the level sets of $\tilde F$ are precisely the Liouville tori spanned by $X_{f_1},\ldots,X_{f_r}$.

Now, as described in~\cite{Laurent-Gengoux2010}(Prop. 3.2) for classical non-commutative integrable systems, choosing an arbitrary Riemannian metric on $M$ defines a canonical projection $\psi: U \to \SF_m$. Setting  $\phi:= \psi \times \tilde F$ we have a commuting diagram
\begin{align}\label{topresult}
  \begin{diagram}
    \node{U}\arrow{e,t,--}{\phi}\arrow{se,b}{\tilde F}\node{\T^r \times B^s}\arrow{s,r}{\pi}\\
    \node[2]{B^{s}}
  \end{diagram}
\end{align}
where
$$\pi=(\pi_1,\ldots ,\pi_s):\T^r \times B^s\to B^s$$
is the canonical projection.

The  change does not affect the Poisson structure on the target space. The commuting diagram \eqref{topresult} implies that
 $$F = \underbrace{(\log|\pi_1|,\pi_2,\ldots, \pi_s)}_{=:\pi'}\circ\phi$$
 so the Poisson structure on the target space $V=F(U)=\pi'(\T^r \times B^s)$ induced by $F$ and $\pi'$ is the same.
 \end{proof}

The upshot is that for the semi-local study of non-commutative $b$-integrable systems around a Liouville torus we can restrict our attention to systems on $(\T^r \times B^s,\omega)$ where $\omega$ is the $b$-symplectic structure induced by the diffeomorphism $\phi$ in the proof above and where the integrals $F=(f_1,\ldots,f_s)$ are given by
$$f_1 = \log|\pi_1|, f_2 = \pi_2,\ldots, f_s = \pi_s,$$
where $\pi_1,\ldots,\pi_s$ are the projections on to the components of $B^s$ and where we assume that the $b$-symplectic structure has exceptional hypersurface $\{ \pi_1  = 0\}$. Also, we can assume that the system is regular on the whole manifold $M=\T^r \times B^s$. We refer to this system as the {\it standard non-commutative $b$-integrable system} on $\T^r \times B^s$.

\begin{remark} The previous result gives a semilocal description of the manifold and the integrals. However, no information is given about the symplectic structure. In contrast, the action-angle coordinate theorem will specify the integrable system with respect to the canonical $b$-symplectic form ($b$-Darboux form) on $\T^r \times B^s$.
\end{remark}

\subsection{Darboux-Carath\'{e}odory theorem}\label{sec:darbouxcaratheodory}

The following is a key ingredient for the proof of the action-angle coordinate theorem. It tells us that we can locally extend a set of independent commuting functions to a $b$-Darboux chart.

\begin{lemma}[{\bf Darboux-Carath\'{e}odory theorem for $b$-integrable systems}]\label{darbouxcarath}
  Let $m$ be a point lying inside the exceptional hypersurface $Z$ of a $b$-symplectic manifold $(M^{2n},\omega)$. Let $t$ be a local defining function of $Z$ around $m$. Let $f_1,\dots,f_{k}$ be a set of commuting $C^\infty$ functions with differentials that are linearly independent at $m$ as elements of $^b T^\ast_m(M)$.
 Then there exist, on a neighborhood $U$ of $m$, functions
  $g_1,\dots,g_k,t, p_2,\ldots, p_{n-k}, q_1,\ldots, q_{n-k}$, such that
  \begin{enumerate}[(a)]
    \item The $2n$ functions $(f_1,g_1, \dots,f_k, g_k, t, q_1, p_1,q_2,\dots,p_{n-k},q_{n-k})$ form a system of coordinates on $U$
    centered at $m$.
    \item The $b$-symplectic form $\omega$ is given on $U$ by
    \begin{equation*}
      \omega=\sum_{i=1}^k df_i \we dg_i + \frac{1}{t} dt \we dq_1 + \sum_{i=2}^{n-k} d p_i \we d q_i.
    \end{equation*}
  \end{enumerate}
\end{lemma}

\begin{proof}
 Let us denote the $b$-Poisson structure dual to $\omega$ by $\Pi$. From the Darboux-Carath\'{e}odory Theorem for non-commutative integrable systems on Poisson manifolds it follows that on a neighborhood $U$ of $m$ we can complete the functions $f_1,\ldots, f_{k}$ to a coordinate system
$$(f_1,g_1, \dots,f_{k}, g_{k}, z_1,\ldots, z_{2n-2r+2})$$
 centred at $m$ such that the $b$-Poisson structure reads
     \begin{equation*}
      \Pi=\sum_{i=1}^{k} \frbd{f_i} \we \frbd{g_i} + \sum_{i,j=1}^{2n-2k} \phi_{ij}(z)\pp{}{z_i}\we\pp{}{z_j}
    \end{equation*}
for some functions $\phi_{ij}$. The image of the coordinate functions is an open subset of  $\R^{2n}$; we can assume that it is a product $U_1 \times U_2$ where $U_2$ corresponds to the image of $z_1,\ldots, z_{2n-2k}$. Then
$$\Pi_2=\sum_{i,j=1}^{2n-2r+2} \phi_{ij}(z)\pp{}{z_i}\we\pp{}{z_j}$$
is a $b$-Poisson structure on $U_2$ and hence by the $b$-Darboux theorem (Theorem \ref{thm:bdarboux}), there exist coordinates on $U_2$
$$(t,q_1,p_2,q_2,\dots,p_{n-k},q_{n-k}),$$
where $t$ is the local defining function for $Z$  that we fixed in the beginning, such that
$$\Pi_2=t \pp{}{t}\we\pp{}{q_1} + \sum_{i=2}^{n-r} \pp{}{p_i}\we\pp{}{q_i}.$$
The result follows immediately.
\end{proof}

\begin{remark} A different proof can be given using the tools of \cite{KMS}. \end{remark}

\subsection{Action-angle coordinates}\label{subsec:aacoo}

Let $(M^{2n},\omega,F)$ be a non-commutative $b$-integrable system of rank $r$. Let  $p\in M_{F,r}\cap\U_\SF$ be a regular point of the system lying inside the critical hypersurface and let $\SF_p$ be the Liouville torus passing through $p$. For a semilocal description of the system around $\SF_p$, by Proposition \ref{prop:standarddonut} we can assume that we are dealing with the ``standard model" of a non-commutative $b$-integrable system, i.e. the manifold is the cylinder $\T^r \times B^s$ with some $b$-symplectic form $\omega$ whose critical hypersurface is $Z=\{\pi_1=0\}=\T^r \times \{0\} \times B^{s-1}$ and the integrals are $f_1=\log|\pi_1|, f_i=\pi_i$, $i=2,\ldots,r$. Let $c$ be the modular period of $Z$.

\begin{theorem}\label{thm:action-anglenc}
Then on a neighborhood $W$ of $\SF_m$ there exist $\R\backslash \Z$-valued smooth functions
  $${\theta_1},\dots,{\theta_r}$$
and  ${\R}$-valued smooth functions
$$t, a_2,\dots, a_r,p_1, \dots, p_ \ell, q_1, \ldots, q_\ell $$
 where $\ell=n-r=\frac{s-r}{2}$ and $t$ is a defining function of $Z$,  such that
  \begin{enumerate}
    \item The functions $(\theta_1,\dots,\theta_r,t,a_2,\dots,a_{r},p_1,\dots,p_{n-r},q_1 \dots, q_{n-r})$ define a diffeomorphism
            $W\simeq\T^r\times B^{s}$.
    \item The $b$-symplectic structure can be written in terms of these coordinates as
    \begin{equation*}
       \omega = \frac{c}{t} d \theta_1 \we dt + \sum_{i=2}^{r} d{\theta_i}\we d{a_i} + \sum_{k=1}^{\ell}d p_k\we d q_k .
    \end{equation*}
    \item The leaves of the surjective submersion $F=(f_1,\dots,f_{s})$ are given by the projection onto the
      second component $\T^r \times B^{s}$, in particular, the functions $f_1,\dots,f_s$ depend on
      $t,a_2,\ldots, a_r, p_1,\dots,p_\ell,q_1 \dots, q_\ell$ only.
  \end{enumerate}
  The functions
$$\theta_1,\dots,\theta_{r}$$ are called \emph{angle coordinates}, the functions
$$t,a_2,\dots,a_r$$
  are called \emph{action coordinates} and the remaining coordinates
$$p_1,\dots,p_{n-r}, q_1, \ldots, q_{n-r}$$
 are called
  \emph{transverse coordinates}.
\end{theorem}

We will need the following two lemmas for the proof of this theorem:

\begin{lemma}\label{lem:per}
Let $F:M\to \overline{R}^s$ be an $s$-tuple of $b$-functions on the $b$-symplectic manifold $M=\T^r\times B^s$. If the coefficients of a vector field of the form $Z=\sum_{j=1}^r\psi_j X_{f_j}$  are $F$-basic and the vector field has period one, then the coefficients are Cas-basic.
\end{lemma}
\begin{proof}
The proof is exactly the same as in \cite{Laurent-Gengoux2010} replacing Hamiltonian by $b$-Hamiltonian vector field.
\end{proof}

The following lemma was proved in \cite{Laurent-Gengoux2010} (see Claim 2),
\begin{lemma}\label{lem:poisson}If $\Y$ is a complete vector field of period one and $P$ is a bivector field for which
$\L_{\Y}^2 P=0$, then $\L_\Y P=0$.
\end{lemma}
We can now proceed with the proof of Theorem \ref{thm:action-anglenc}:

\begin{proof} (of Theorem \ref{thm:action-anglenc})
In the first step we perform ``uniformization of periods" similar to \cite{Laurent-Gengoux2010} and \cite{KMS}. The joint flow of the vector fields $X_{f_1},\ldots,X_{f_r}$ defines an $\R^r$-action on $M$, but in general not a $\T^r$-action, although it is periodic on each of its orbits $\T^r\times \{\const\}$.

Denoting the time-$s$ flow of the Hamiltonian vector field $X_{f}$ by $\Phi_{X_{f}}^{s}$, the joint flow of the Hamiltonian vector fields $X_{f_1}, \dots, X_{f_r}$ is
\begin{align*}
\Phi :\R^r \times ( \T^r \times B^s ) &\to \T^r \times B^s \\
 \big((s_1,\ldots,s_r),(x, b)\big) &\mapsto \Phi_{X_{f_1}}^{s_1}\circ\dots\circ \Phi_{X_{f_r}}^{s_n}(x, b).
\end{align*}
Because the $X_{f_i}$ are complete and commute with one another, this defines an $\R^r$-action on $\T^r \times B^s$. When restricted to a single orbit $\T^r \times \{b\}$ for some $b \in B^s$, the kernel of this action is a discrete subgroup of $\R^r$, hence a lattice $\Lambda_{b}$, called the {\it period lattice} of the orbit $\T^r \times \{b\}$. Since the orbit is compact, the rank of $\Lambda_b$ is $r$. We can find smooth functions (after shrinking the ball $B^s$ if necessary)
$$\lambda_i: B^s \rightarrow \mathbb{R}^r, \quad i=1,\ldots,r$$
such that
\begin{itemize}
\item $(\lambda_1(b), \lambda_2(b), \dots, \lambda_r(b))$ is a basis for the period lattice $\Lambda_b$ for all $b \in B^s$
\item $\lambda_i^1$ vanishes along $\{0\} \times B^{s-1}$ for $i > 1$, and $\lambda_1^1$ equals the modular period $c$ along $\{0\} \times B^{s-1}$. Here, $\lambda_i^{j}$ denotes the $j^{\textrm{th}}$ component of $\lambda_i$.
\end{itemize}

Using these functions $\lambda_i$ we define the ``uniformized'' flow
\begin{align*}
\tilde \Phi :\R^r \times ( \T^r \times B^s ) &\to ( \T^r \times B^s )\\
 \big((s_1,\ldots,s_r),(x, b)\big) &\mapsto \Phi\big(\sum_{i=1}^r s_i \lambda_i(b), (x,b) \big).
\end{align*}
The period lattice of this $\R^r$-action is constant now (namely $\Z^r$) and hence the action naturally defines a $\T^r$ action. In the following we will interpret the functions $\lambda_i$ as functions on $\T^r\times B^s$ (instead of $B^s$) which are constant on the tori  $\T^r\times \{b\}$.

We denote by $Y_1,\ldots,Y_r$ the fundamental vector fields of this action. Note that $Y_i = \sum_{j=1}^r \lambda_i^j X_{f_j}$.
We now use the Cartan formula for $b$-symplectic forms (where the differential is the one of the complex of $b$-forms \cite{Guillemin2012}  \footnote{The decomposition of a $b$-form of degree $k$ as $\omega=\frac{dt}{t}\wedge \alpha+\beta$ for  $\alpha, \beta$ De Rham forms proved in \cite{Guillemin2012} allows to extend the Cartan formula valid for smooth De Rham forms to  $b$-forms.}) to compute the following expression:
\begin{align} \label{eqn:poisson}
\mathcal{L}_{Y_i}\mathcal{L}_{Y_i}\omega &= \mathcal{L}_{Y_i}(d(\iota_{Y_i}\omega)+\iota_{Y_i} d\omega)\\
&= \mathcal{L}_{Y_i}(d(-\sum_{j=1}^n \lambda_i^j d{f_j})) \\
&= -\mathcal{L}_{Y_i}\left(\sum_{j=1}^n d\lambda_i^j \wedge d{f_j}\right) = 0
\end{align}

\noindent where in the last equality we used the fact that $\lambda_i^j$ are constant on the level sets of $F$. By applying Lemma \ref{lem:poisson} this yields  $\mathcal{L}_{Y_i}\omega = 0$, so the vector fields $Y_i$ are Poisson vector fields, i.e. they preserve the $b$-symplectic form.

We now show that the $Y_i$ are Hamiltonian, i.e. the ($b$-)one-forms
\begin{equation}\label{alphai}
\alpha_i:=\iota_{Y_i}\omega = -\sum_{j=1}^r \lambda_i^j df_j,\quad i=1,\ldots,r,
\end{equation}
which are closed (because  $Y_i$ are Poisson) have a ($^b C^\infty$-)primitive $a_i$. Since $\lambda_i^1$ vanishes along $\T^r \times \{0\} \times B^{s-1}$ for $i > 1$, the one-forms $\alpha_i$ defined in Equation \eqref{alphai} and hence the functions $a_i$ are smooth for $i>1$. On the other hand, $\lambda_1^1$ equals the modular period $c$ along $\T^r \times \{0\} \times B^{s-1}$ and therefore $a_1= c\log|t|$ for some defining function $t$.

We compute the functions $a_2,\ldots,a_r$ explicitly by applying a homotopy formula to the smooth one-forms $\alpha_2,\ldots,\alpha_r$. This not only yields that these one-forms are exact but moreover that their $C^\infty$-primitives  $a_2,\ldots,a_r$ are Cas-basic. (For the $b$-function $a_1= c\log|t|$ this is clear.)  This is equivalent to proving that these closed forms are exact for the corresponding sub-complex of Cas-basic $b$-forms. We do this by means of adapted homotopy operators.

 Consider the following homotopy formula (see for instance  \cite{evaromero}):
$$ \alpha_i - \phi^\ast_0(\alpha_i)= I(\underbrace{d(\alpha_i)}_{=0})+d(I(\alpha_i)),\quad i=2,\ldots,r$$
where the functional $I$ will be defined below and $\phi_\tau$ is the retraction from $\T^r\times B^s$ to $\T^r\times \{0\} \times B^{s-r}$:
$$\phi_\tau(x_1,\ldots,x_r,b_1,\ldots,b_r,b_{r+1},\ldots,b_s)=(x,\tau b_1,\ldots \tau b_r, b_{r+1},\ldots,b_s).$$
Note that $\phi^\ast_0(\alpha_i)=0$ since for any vector field $X\in \mathcal{X}(\T^r\times \{0\} \times B^{s-r})$ we have $\alpha_i(X)=0$. Recall that $\alpha_i$ is a linear combination of $d\pi_2,\ldots,d\pi_r$ and therefore evaluates to zero for $X$ a linear combination of $\pp{}{x_1},\ldots,\pp{}{x_r},\pp{}{\pi_{r+1}},\ldots,\pp{}{\pi_{s}}$. Therefore the homotopy formula tells us that the Hamiltonian function of $\alpha_i$ ($i=2,\ldots,r$) is explicitly given by $I(\alpha_i)$, which is defined as follows:
$$I(\alpha_i)=\int_0^1 \phi_\tau^\ast(\iota_{\xi_\tau}(\alpha_i)).$$
 Here $\xi_\tau$  is the vector field associated with the retraction:
$$\xi_\tau= \frac{d\phi_\tau}{d\tau}\circ \phi_\tau^{-1} = \frac{1}{\tau} \sum_{k=1}^s  \pi_k \pp{}{\pi_k}.$$
 Therefore we have
$$\iota_{\xi_\tau}(\alpha_i) = \frac{1}{\tau} \sum_{j=2}^r \lambda_i^j d\pi_j(\xi_\tau) = \frac{1}{\tau} \sum_{j=2}^r\sum_{k=1}^s \lambda_i^j \pi_k d{\pi_j}\left(\pp{}{\pi_k}\right)= \frac{1}{\tau} \sum_{j=2}^r \lambda_i^j \pi_j .$$
In the last equality we have used $d\pi_j(\pp{}{\pi_k})=\delta_{jk}$ for $j>2$.

The projections $\pi_j,j=1,\ldots,r$, are obviously Cas-basic. The functions $\lambda_i^j$ are Cas-basic by Lemma \ref{lem:per}. The pullback $\phi_\tau^\ast$ does not affect the Cas-basic property since it leaves the non-commutative part of the system invariant. We conclude that the functions $\phi_\tau^\ast (\iota_{\xi_\tau}(\alpha_i))$  and hence $a_1,\ldots,a_r$ are Cas-basic.

We apply the Darboux-Carath\'{e}odory theorem for $b$-integrable systems to a point $p\in \T^r\times \{0\}$ and the independent commuting smooth functions $a_2,\ldots,a_n$. Then on a neighborhood $U$ of $p$ we obtain a set of coordinates $(t, g_1, a_2, g_2 ,\ldots,a_r,g_r,q_1,p_1,q_2,p_2,\ldots,q_\ell,p_\ell)$, where $\ell = (s-2r)/2$, such that
\begin{equation}\label{omegau}
      \omega|_U=\frac{c}{t} dt \we dg_1 + \sum_{i=2}^k da_i \we dg_i +  \sum_{i=1}^{\ell} d p_i \we d q_i.
\end{equation}
The idea of the next steps is to extend this local expression to a neighborhood of the Liouville torus using the $\T^r$-action given by the vector fields $X_{a_k}$. First, note that the functions $(q_1,p_1,q_2,p_2,\ldots,q_\ell,p_\ell)$ do not depend on $f_i$ and therefore can be extended to the saturated neighborhood $W:=\pi^{-1}(\pi(U))$. Note that $Y_i = \pp{}{g_i}$ and therefore the flow of the fundamental vector fields of the $Y_i$-action corresponds to translations in the $g_i$-coordinates. In particular, we can naturally extend the functions $g_i$ to the whole set $W$ as well.

We want to see that the functions
\begin{equation}\label{ourcoo}
t, g_1, a_2, g_2 ,\ldots,a_r,g_r, q_1,p_1,q_2,p_2,\ldots,q_\ell,p_\ell
\end{equation}
which are defined on $W$, indeed define a chart there (i.e. they are independent) and that $\omega$ still has the form given in Equation \eqref{omegau}.

It is clear that $\{a_i, g_j\}=\delta_{ij}$ on $W$. To show that $\{g_i, g_j\}=0$, we note that this relation holds on $U$ and flowing with the vector fields $X_{a_k}$ we see that it holds on the whole set $W$:
$$ X_{a_k}\big(\{g_i,g_j\} \big) = \{\{g_i, g_j\},  a_k\} = \{ g_i, \delta_{ij}\} - \{g_j,\delta_{ik}\} = 0. $$
This verifies that $\omega$ has the form \eqref{omegau} above and in particular, we conclude that the derivatives of the functions \eqref{ourcoo} are independent on $W$, hence these functions define a coordinate system.

Since the vector fields $\pp{}{g_i}$ have period one, we can view $g_1,\ldots,g_r$ as $\R\backslash \Z$-valued functions (``angles") and therefore use the letter $\theta_i$ instead of $g_i$.
\end{proof}

 \begin{remark} In the language of cotangent models introduced in \cite{km}, this theorem can be expressed as saying that a non-commutative $b$-integrable system is semilocally equivalent given by the  the twisted $b$-cotangent lift of the $\T^r$-action on itself by translations.
 \end{remark}

\end{document}